\documentclass[a4paper,12pt,reqno]{amsart}
\usepackage[margin=1in]{geometry}
\usepackage{amsmath}
\usepackage{amsthm}
\usepackage{amssymb}
\usepackage{graphicx}
\usepackage{parskip}
\usepackage[pdftex]{hyperref}
\hypersetup{breaklinks=true,colorlinks=true,citecolor=green,urlcolor=red,linkcolor=blue,pdfpagemode=UseNone}

\theoremstyle{plain}
\newtheorem{theorem}{Theorem}[section]
\newtheorem{lemma}[theorem]{Lemma}

\DeclareMathOperator{\Var}{Var}
\DeclareMathOperator{\Aut}{Aut}
\DeclareMathOperator{\cl}{cl}
\DeclareMathOperator{\Cl}{Cl}

\begin{document}

\title{Diameter Thresholds of Random Cayley Graphs}
\author[D. Christofides]{Demetres Christofides$^{1}$}
\author[K. Markstr\"om]{Klas Markstr\"om$^{2}$}
\author[C. Savvidou]{Christina Savvidou$^{1}$}

\thanks{\noindent
\begin{minipage}[t]{0.95\textwidth}
$^{1}$INSPIRE Research Centre and UCLan Cyprus, 7080 Pyla, Larnaka, Cyprus.\\
\textit{Email addresses:}
\texttt{dchristofides@uclan.ac.uk},
\texttt{csavvidou@uclan.ac.uk}.\\
$^{2}$Department of Mathematics and Mathematical Statistics,
Ume{\aa} University, Ume{\aa}, Sweden.\\
\textit{Email address:}
\texttt{klas.markstrom@umu.se}. \\[4pt]
This work was supported by the project EXCELLENCE/0524/0534 implemented under the programme of social cohesion ``THALIA 2021-2027'' co-funded by the European Union, through Research and Innovation Foundation. 
\end{minipage}
}
\date{\today}

\begin{abstract}
Given a group $G$, the model $\mathcal{G}(G,p)$ denotes the probability space of all Cayley graphs of $G$ where each element of $G$ is included in the generating set independently at random with probability $p$.

In this article, we investigate the threshold probabilities for the diameter of random graphs in this model. Specifically, let $d_N = (1-\gamma)\sqrt{\frac{\log{N}}{2\log{\log{N}}}}$, where $\gamma \in (0,1)$ is any fixed real number. We show that for any $\varepsilon > 0$, any family of groups $G_k$ of order $N_k$ for which $N_k \to \infty$, and any integer $2 \leqslant d\leqslant d_{N_k}$, a graph $\Gamma_k \in \mathcal{G}(G_k,p)$ with high probability has diameter at most $d$ if $p \geqslant \sqrt[d]{(1+\varepsilon) d! \frac{\log{N_k}}{N_k^{d-1}}}$, and diameter greater than $d$ if $p \leqslant \sqrt[d]{\frac{1-\varepsilon}{2^d} \frac{\log{N_k}}{N_k^{d-1}}}$. 

Up to a constant factor, these thresholds are similar to those for the usual Erd\H{o}s-R\'enyi random graphs. However, the precise thresholds in our model depend on the underlying family of groups. We provide specific examples of group families demonstrating that both of our bounds are best possible.
\end{abstract}

\maketitle
\section{Introduction}

Let us begin by recalling that given a group $G$ and a subset $S$ of $G$, the (undirected) Cayley graph $\Gamma = \Gamma(G;S)$ of $G$ with respect to $S$ has the elements of $G$ as its vertex set and has an edge between $g$ and $h$ if and only if $g^{-1}h \in S$ or $h^{-1}g \in S$. In other words, there is an edge between $g$ and $h$ if and only if there is an $s \in S$ with $gs=h$ or with $hs = g$. We ignore any loops or multiple edges. In particular, whether $1 \in S$ or not is immaterial. Observe, for example, that $\Gamma$ is connected if and only if the set $S$ generates the group $G$. Throughout the paper, we will often refer to the set $S$ as the generating set of the graph $\Gamma$ regardless of whether it is a generating set for the group $G$ or not. 

The model $\mathcal{G}(G,p)$ is the probability space of all graphs $\Gamma(G;S)$ in which every element of $G$ is assigned to the set $S$ independently at random with probability $p$. This model has many similarities with the model $\mathcal{G}(n,p)$, which is the probability space of all graphs with vertex set $\{1,2,\ldots,n\}$ in which every edge appears independently with probability $p$. We refer the reader to~\cite{CM-Latin} for some of these similarities. There are however many differences between these two models. An obvious difference is that every graph in $\mathcal{G}(G,p)$ is regular while with high probability this is not the case in the model $\mathcal{G}(n,p)$ (unless in the trivial cases in which $p$ is either so large or so small that with high probability forces the graph to be complete or empty respectively). This difference also motivates the comparison of the model $\mathcal{G}(G,p)$ with the model $\mathcal{G}_{n,r}$, the probability space of all $r$-regular graphs on $\{1,2,\ldots,n\}$ taken with the uniform measure. These two models still have significant differences. For example, every graph $\Gamma \in \mathcal{G}(G,p)$ is not only regular, but in fact it has a high degree of symmetry. More specifically, every element $g$ of $G$ defines an automorphism of $\Gamma$ by left multiplication and so $G$ is a subgroup of $\Aut(\Gamma)$. On the other hand, it is known that for every $3 \leqslant r \leqslant n-4$, graphs in $\mathcal{G}_{n,r}$ have with high probability a trivial automorphism group~\cite{KSV}. Given the success of random graphs in settling many graph theory questions and the fact that random Cayley graphs have some important properties not shared by other random graph models, we see that their study is highly desirable. 

In this paper we study the diameter of random Cayley graphs. In particular, for every fixed integer $d \geqslant 2$ we will be concerned with the range of $p$ for which the diameter of the random Cayley graph changes from greater than $d$ to at most $d$. It is well-known (see e.g.~\cite{Bollobas}) that for any $\varepsilon > 0$, a graph from $\mathcal{G}(n,p)$ with high probability has diameter greater than $d$ if $p \leqslant \sqrt[d]{\frac{(2 - \varepsilon) \log{n}}{n^{d-1}}}$ and diameter at most $d$ if $p \geqslant \sqrt[d]{\frac{(2 + \varepsilon) \log{n}}{n^{d-1}}}$. For the case of random regular graphs, exactly the same asymptotic thresholds hold. This follows from a result of Gao, Isaev and McKay~\cite{GIM} which settled the sandwich conjecture of Kim and Vu~\cite{KV} in a suitable range of the parameters. It is worth noting that the conjecture has now been settled in full by Behague, Il'kovic and Montgomery~\cite{BIM}. While this sandwiching establishes a strong universality between random regular graphs and $\mathcal{G}(n,p)$, our results will highlight that such strict universality breaks down for random Cayley graphs, where local algebraic dependencies alter the precise threshold constants.

There are several existing results which directly or indirectly bound the diameter of a random Cayley graph.  In \cite{CM-Expansion} Christofides and Markstr\"om showed that random Cayley graphs with a logarithmic number of generators are with high probability good expander graphs. As for all expanders, this implies that the diameter is at most logarithmic in the number of vertices.  With fewer generators, some groups do not get connected Cayley graphs so this result is to leading order sharp. This expansion property also implies that these Cayley graphs are with high probability Hamiltonian \cite{KS-Hamiltonian,BDMP-Cayley-Lovasz}.  For Cayley graphs with degree above $n^{1-c}$, for some $c>0$, this holds for all Cayley graphs, and for degree above $\alpha n$, for any $\alpha>0$, the Hamiltonian cycle can be found in polynomial time \cite{Christofides-Hladky-Mathe-Hamilton}. Instead, looking at low diameters, \cite{CM-diam} determined the threshold(s) for  diameter 2. For general random Cayley graphs, the existing results thus leave the threshold behaviour for diameters from 3 up to a logarithmic order open.  Here we should also mention the following conjecture of Babai:  There exists a constant $c>0$ such that for any non-abelian simple group $G$, and any generating set $S$ such that the Cayley graph is connected, the diameter is at most $(\log |G|)^c$. Several recent papers have shown that this holds for different group families, e.g.  \cite{Breuillard-Green-Tao-Approximate-Subgroups,Helfgott-Seress-Permutation-Diameter,Garonzi-Halasi-Somlai-Classical-Transvection,Halasi-SLnp-Transvection,Pham-Zhang-Logarithmic-Diameter}. A result of  particular interest is given in \cite{Eberhard-Jezernik-Babai-Random-Generators} where it is shown that, for certain groups, if a generating set makes  the Cayley graph connected  and we add a constant number of random generators then the new generating set satisfies the bound from Babai's conjecture.  This type of random perturbation result for diameter would be interesting to pursue for lower diameters as well.

For some group families specialized diameter and distance results have been found. Abelian groups have been studied both for their structural simplicity and close connection to number theory.  In \cite{Marklof-Frobenius,Marklof-Strombergsson-Circulant-Diameters} it was shown that for a given number of generators, the diameters converge in distribution for certain families of abelian groups. In \cite{Hermon-OleskerTaylor-Abelian-Cayley} this was extended to all abelian groups, and that the typical distance between two vertices also becomes concentrated around a specific value.  Going one step beyond the abelian case, similar distribution results were shown for nilpotent groups in \cite{ElBaz-Pagano-Nilpotent-Diameters}. In  \cite{Sardari-Ramanujan-Cayley-Diameters} the diameters of  certain algebraic Ramanujan graphs were studied and compared to those for random 6-regular Cayley graphs of $\operatorname{PSL}_2(\mathbb Z/(q\mathbb Z))$. For the latter, the diameter was conjectured to  asymptotically be $\log_5(N)$. 

Note that when speaking about random Cayley graphs, the group $G$ and thus its order is fixed. So strictly speaking, it does not really make sense to speak of events occurring with high probability. Instead, we have to consider a family of groups $G_k$ for which their orders tend to infinity. Motivated by the above results, it is natural to conjecture that for any $\varepsilon > 0$, if $(G_k)$ is a family of groups having order $n_k$ with $n_k \to \infty$, then there is a constant $c$ such that if $\Gamma \in \mathcal{G}(G_k,p)$ then with high probability $\Gamma$ has diameter greater than $d$ if $p \leqslant  \sqrt[d]{(c - \varepsilon)\frac{\log{n_k}}{n_k^{d-1}}}$ and diameter at most $d$ if $p \geqslant  \sqrt[d]{(c + \varepsilon)\frac{\log{n_k}}{n_k^{d-1}}}$. It turns out that this naive conjecture is wrong for the following reason: We will see that the conjecture is true for several natural families of groups. However, different families of groups can give rise to different constants $c$. In particular, if we take two families $(G_k)$ and $(H_k)$ for which the conjecture is true but with different values of $c$, then if we interlace these families into a new one we get a family for which this naive conjecture is false. This suggests that the right conjecture to consider is the following: For any $\varepsilon > 0$, if $(G_k)$ is a family of groups having order $n_k$ with $n_k \to \infty$, then there are constants $c_1,c_2$ such that if $\Gamma \in \mathcal{G}(G_k,p)$ then with high probability $\Gamma$ has diameter greater than $d$ if $p \leqslant  \sqrt[d]{(c_1 - \varepsilon)\frac{\log{n_k}}{n_k^{d-1}}}$ and diameter at most $d$ if $p \geqslant  \sqrt[d]{(c_2 + \varepsilon)\frac{\log{n_k}}{n_k^{d-1}}}$. Our aim in this paper is to prove this conjecture. Moreover, we will show that the values we obtain for $c_1$ and $c_2$ are best possible. In particular, the following results extend the previous results of Christofides and Markstr\"om~\cite{CM-diam} for the case $d=2$ to all constant $d$, and much of the possible range for growing $d$.

To state our results rigorously without explicitly carrying the subscript $k$ for group families, all statements involving the phrase ``with high probability'' should be understood as asserting that the respective error probabilities tend to $0$ as $N \to \infty$, uniformly over the class of groups under consideration. 

In fact, our results are going to work for growing functions of the diameter as well. In particular, for groups of size $N$, the following results are valid with diameters of size at most $d_N$, where
\[ d_N = (1-\gamma)\sqrt{\frac{\log{N}}{2\log{\log{N}}}},\]
and $\gamma \in (0,1)$ is any fixed real number. With this choice of $d_N$, we point out for later use, that
\[ 4d_N^2\log{d_N} \sim (1-\gamma)^2\log{N}.\]

\begin{theorem}\label{T:<=D}
    Let $\varepsilon > 0$, $d,N \in \mathbb{N}$ with $2 \leqslant d \leqslant d_N$, and let $G$ be any group of order~$N$. If $\Gamma \in \mathcal{G}(G,p)$, where $p \geqslant \sqrt[d]{d!(1 + \varepsilon)\frac{\log{N}}{N^{d-1}}}$, then with high probability, the diameter of $\Gamma$ is at most $d$.
\end{theorem}

In the other direction, we have the following bound:

\begin{theorem}\label{T:D>=3}
    Let $\varepsilon > 0$, $d,N \in \mathbb{N}$ with $2 \leqslant d \leqslant d_N$, and let $G$ be any group of order~$N$. If $\Gamma \in \mathcal{G}(G,p)$, where $p \leqslant \sqrt[d]{\frac{1 - \varepsilon}{2^d} \cdot \frac{\log{N}}{N^{d-1}}}$, then with high probability, the diameter of $\Gamma$ is greater than $d$.
\end{theorem}

For abelian groups, we can improve this result to the following:

\begin{theorem}\label{T:D>=3-Abelian}
    Let $\varepsilon > 0$, $d,N \in \mathbb{N}$ with $2 \leqslant d \leqslant d_N$, and let $G$ be any abelian group of order~$N$. If $\Gamma \in \mathcal{G}(G,p)$, where $p \leqslant \sqrt[d]{\frac{d!(1 - \varepsilon)}{2^d} \cdot \frac{\log{N}}{N^{d-1}}}$, then with high probability, the diameter of $\Gamma$ is greater than $d$.
\end{theorem}

All three of the above results are best possible, as exhibited by the results below:

\begin{theorem}\label{E:D>=3}
   Let $\varepsilon > 0$, $n \in \mathbb{N}$ and $G = \mathbb{Z}_2^n$. Let $N = 2^n = |G|$ and suppose $d \in \mathbb{N}$ with $2 \leqslant d \leqslant d_N$. If $\Gamma \in \mathcal{G}(G,p)$, where  $p \leqslant \sqrt[d]{d!(1 - \varepsilon)\frac{\log{N}}{N^{d-1}}}$,  then with high probability, the diameter of $\Gamma$ is greater than $d$.
\end{theorem}

\begin{theorem}\label{E:D>=3-Abelian}
    Let $\varepsilon > 0$, $d,N \in \mathbb{N}$ with $2 \leqslant d \leqslant d_N$, and let $G$ be the cyclic group of order~$N$. If $\Gamma \in \mathcal{G}(G,p)$, where $p \geqslant \sqrt[d]{\frac{d!(1 + \varepsilon)}{2^d} \cdot \frac{\log{N}}{N^{d-1}}}$, then with high probability, the diameter of $\Gamma$ is at most $d$.
\end{theorem}

Before stating the corresponding result for general groups, we recall that for an element $x$ of a group $G$, its conjugacy class is the set 
\[\Cl(x) = \{y^{-1}xy:y \in G\}.\] 
We will write $\cl(x)$ for the size of the conjugacy class of $x$. We also recall that an element $x$ of $G$ with $x^2=1$ is called an involution.

\begin{theorem}\label{E:D=d}
Let $\varepsilon > 0$ be small enough, $d,N \in \mathbb{N}$ with $2 \leqslant d \leqslant d_N$, and 
\[M = \exp\left \{2\sqrt{\log{N}\log{\log{N}}}\right\}.\] 
Let $G$ be a group of order $N$ such that \begin{itemize}
\item[(a)] $G$ contains at most $N^{(2 + \varepsilon)/4}$ involutions.
\item[(b)] $G$ contains at most $N^{1/d}$ elements $x$ with $\cl(x) \leqslant M$.
\item[(c)] $G$ contains at most $N^{1/2d}$ involutions $x$ with $\cl(x) \leqslant M$.
\end{itemize}
If $\Gamma \in \mathcal{G}(G,p)$, where $p \geqslant \sqrt[d]{\frac{1 + \varepsilon}{2^d} \cdot \frac{\log{N}}{N^{d-1}}}$, then with high probability, the diameter of $\Gamma$ is at most $d$.
\end{theorem}

To provide some combinatorial intuition behind the constants $2^d$ and $d!$ appearing in the theorems above, let us briefly describe the underlying algebraic mechanisms: When a generator $s \in S$ is not an involution, the edge between $g$ and $gs$ is exposed twice in the random selection (either via $s$ or $s^{-1}$), meaning that an edge is present with probability $2p-p^2 \approx 2p$. For groups with few involutions, this effective doubling of edge probabilities over paths of length $d$ yields a factor of $(2p)^d = 2^d p^d$, lowering the required threshold by $2^d$. Conversely, in abelian groups, products commute, meaning that all $d!$ permutations of a set of $d$ distinct generators collapse into the same terminal vertex. Consequently, the neighbourhoods expand more slowly, which explains why the threshold probability increases by a factor of $\sqrt[d]{d!}$. 

The growing value of $M$ in Theorem~\ref{E:D=d} is only really needed for the growing diameter $d$. For every fixed $d$, the value of $M$ can be chosen to be a polynomial in $1/\varepsilon$ as long as $\varepsilon$ is small enough, thus increasing the number of families of groups for which it holds. In fact, one should check that there is at least one such family. Even though the symmetric groups do satisfy these properties, the simplest family that we could find is the following:

Fix a prime $p \equiv 3 \bmod 4$ and consider the group $G$ (under composition) of all functions $f_{a,b}:\mathbb{F}_p \to \mathbb{F}_p$ defined by $f_{a,b}(x) = ax+b$ where $a$ is a quadratic residue modulo $p$ and $b \in \mathbb{F}_p$. Then $|G| = p(p-1)/2$ is odd so $G$ has no involutions. We claim that if $b_1$ and $b_2$ are both quadratic residues or both quadratic non-residues then $f_{a,b_1}$ and $f_{a,b_2}$ are conjugates. Indeed, pick a quadratic residue $a'$ such that $b_2 = a'b_1$. Then 
\[ f_{a',0}f_{a,b_1}f^{-1}_{a',0}(x) = f_{a',0}f_{a,b_1}(x/a') = f_{a',0}(ax/a'+b_1) = ax+a'b_1 = f_{a,b_2}(x). \]
Also $f_{a,0}$ with $a\neq 1$ is conjugate to $f_{a,b}$ for at least one $b \neq 0$. Indeed,
\[ f_{1,b}f_{a,0}f^{-1}_{1,b}(x) = f_{1,b}(ax-ab) = ax-ab+b = f_{a,(1-a)b}\]
and since $a \neq 1$, then $(1-a)b \neq 0$ as long as we choose $b \neq 0$. Therefore every non-trivial conjugacy class has size at least $(p-1)/2 \geqslant \sqrt{N/3}$ which is larger than $M$ when $N$ is large enough. 

As already mentioned, the specific case of diameter $d=2$ was previously resolved by two of the three authors in \cite{CM-diam}.  However, the approach utilized there relied heavily on ad-hoc combinatorial counting and complicated 
total probability arguments tailored strictly to pairs. Such methods do not easily scale to larger  diameters without an explosion in technical complexity, as evidenced by the length of the arguments  required just for $d=2$.

In contrast, the present work introduces a significantly more robust, unified probabilistic framework. Unlike standard random graphs where paths expand independently, random Cayley graphs are plagued by structural dependencies and forced intersections arising from underlying group relations. By leveraging an optimized dependency graph construction and a refined application of Janson's inequality directly on intersecting unordered paths, we completely bypass the technical bottlenecks of the previous approach. This new perspective not only allows us to resolve the general diameter $d$ threshold, even for growing diameters $d < d_N =(1-\gamma)\sqrt{\frac{\log N}{2\log \log N}}$, but it does so in a fundamentally more streamlined and elegant manner. Indeed, our general framework yields a considerably shorter and cleaner proof even when restricted back to the $d=2$ case, illuminating the core algebraic and probabilistic mechanisms at play.

We now give an overview of the structure of the paper. In Section~2 we collect the probabilistic and representation theoretic tools that we will need. In Section~3 we introduce the diameter-$d$ graphs and their dependency graph and state various relevant observations related to these graphs. In Section~4 we prove Theorems~\ref{T:D>=3}, \ref{T:D>=3-Abelian}, and~\ref{E:D>=3}. Section~5 is devoted to the proofs of Theorems~\ref{T:<=D} and~\ref{E:D>=3-Abelian}. In Section~6 we prove Theorem~\ref{E:D=d}. Sections~4 and~5 can be read independently, while in Section~6 we will use the notation and ideas introduced in Section~5. 

\section{Main Tools}

Here we list the main tools that we will use in the proofs of our results.

\subsection{Probabilistic Tools}

These are standard probabilistic results. We refer the reader to~\cite{AS} for their proofs.

\begin{theorem}[Markov's inequality]
Let \(X\) be a non-negative random variable. Then, for any \(a>0\),
\[
    \Pr(X\geqslant a)\leqslant \frac{\mathbb{E}X}{a}.
\]
\end{theorem}

\begin{theorem}[Chebyshev's Inequality]
Let $X$ be a random variable which takes values on the non-negative integers and suppose that it has finite variance and expectation. If $\mathbb{E} X>0$, then
\[
\Pr(X = 0) \leqslant \frac{\Var(X)}{(\mathbb{E} X)^2}.
\]
\end{theorem}

\begin{theorem}[Kleitman's Inequality]
Let $\Omega$ be a finite set and let $\{F_i\}_{i \in I}$ be subsets of $\Omega$, where $I$ is a finite index set. Let $R$ be a random subset of $\Omega$, under the standard product measure, and for each $i \in I$ let $E_i$ be the event that $F_i \subseteq R$. Then 
\[
\Pr\left(\bigcap_{i\in I} \overline{E_i} \right) \geqslant \prod_{i \in I} \Pr(\overline{E_i}).
\]
\end{theorem}

\begin{theorem}[Janson's Inequality]
Let $\Omega$ be a finite set and let $\{F_i\}_{i \in I}$ be subsets of $\Omega$, where $I$ is a finite index set. Let $R$ be a random subset of $\Omega$ where each element of $\Omega$ is chosen independently with some probability (which may depend on the element), and for each $i \in I$ let $E_i$ be the event that $F_i \subseteq R$. Then
\[
\Pr\left(\bigcap_{i\in I} \overline{E_i} \right) \leqslant \exp\left(-\sum_{i \in I} \Pr(E_i) + \frac{1}{2}\sum_{i \in I}\sum_{\{j \neq i : F_i \cap F_j \neq \emptyset\}} \Pr(E_i \cap E_j) \right).
\]
\end{theorem}

\subsection{Group Theoretic Tools}

Recall that for an element $x$ of a group $G$ we denote the size of its conjugacy class by $\cl(x)$. We also write $\cl(G)$ for the number of conjugacy classes of $G$.

By Orbit-Stabilizer Theorem, it is easy to deduce the following standard result:

\begin{theorem}\label{commuting_pairs}
    Let $G$ be a group of order $n$. Then, for every element $x$ of $G$,
    \begin{itemize}
        \item[(a)] there are $n/\cl(x)$ elements $y$ such that $y^{-1}xy=x$ and
        \item[(b)] there are at most $n/\cl(x)$ elements $y$ such that $y^{-1}xy=x^{-1}$. 
    \end{itemize}
\end{theorem}

The following representation theoretic result was proved and used in~\cite{CM-diam} using the Frobenius-Schur indicator. We are not aware of any earlier proof of it.
\begin{theorem}\label{main-rep}
    Let $G$ be a group of order $n$ and let $x$ be an element of $G$. Then there are at most $\sqrt{n\cl(G)}$ elements $y$ of $G$ such that $y^2 = x$. 
\end{theorem}

We will apply Theorem~\ref{main-rep} by having a bound on the number of elements of $G$ which have a `small' conjugacy class size. The following result is a simple group theoretic exercise which shows how this leads to a `small' number of conjugacy classes.

\begin{lemma}\label{cl(G)-bound}
If $G$ is a group of $n$ elements such that at most $m$ elements $x$ of $G$ satisfy $\cl(x) \leqslant M$, then $\cl(G) \leqslant m + n/M$.    
\end{lemma}

\begin{proof}
Since
\[ \sum_{y \in \Cl(x)} \frac{1}{\cl(x)} = 1,\]
then
\[ \cl(G) = \sum_{x \in G} \frac{1}{\cl(x)} \leqslant m + \sum_{\substack{x \in G \\ \cl(x) > M}} \frac{1}{\cl(x)} \leqslant m + \frac{n}{M}. \qedhere\]
\end{proof}

\subsection{Other Tools}

We will use the following result from analysis.
\begin{lemma}\label{exponential asymptotic}
If $(a_n) \to 0$ and $(a_n^2b_n) \to 0$, with $a_n<1$ and $0\leqslant b_n$, then $(1-a_n)^{b_n} \sim \exp\left\{-a_nb_n\right\}$.
\end{lemma}

\begin{proof}
We have
\[ \log(1-a_n) = -a_n - \frac{a_n^2}{2} + O(a_n^3)\]
so the existing conditions guarantee that
\[ b_n\log(1-a_n) = -a_nb_n + o(1).\]
The result follows since $(1-a_n)^{b_n} = e^{b_n\log(1-a_n)}$.
\end{proof}

\section{The diameter-\texorpdfstring{$d$}{d} graphs and their dependency graph}

Let $S$ be a subset of $G$ and let $x \in G \setminus \{1\}$. Then $1=x_0,x_1,\ldots,x_d=x$ is a path of length $d$ from $1$ to $x$ in $\Gamma(G,S)$ if and only if the $x_i$'s are distinct, and for each $i \in \{1,2,\ldots,d\}$ we have that $x_{i-1}^{-1}x_i \in S$ or $(x_{i-1}^{-1}x_i)^{-1} = x_i^{-1}x_{i-1} \in S$. Equivalently, we have a path of length $d$ from $1$ to $x$ in $\Gamma(G,S)$ if and only if we can find $g_1,\ldots,g_d \in S$ and $a_1,\ldots,a_d \in \{-1,1\}$ such that 
\[ g_1^{a_1} \cdots g_d^{a_d} = x\]
with all $1,g_1^{a_1},g_1^{a_1}g_2^{a_2},\ldots,g_1^{a_1} \cdots g_d^{a_d} = x$ being distinct. Note that even though the $x_i$'s are distinct, the $g_i$'s might not be so. 

For $\Gamma(G,S)$ to have diameter greater than $d$, we want, for at least one $x \in G \setminus \{1\}$, to avoid all paths of length $d$ or less from $1$ to $x$.

To this end, we define the following diameter-$d$ graphs:

For each $x\in G\setminus\{1\}$, define $\Gamma_x$ to be the hypergraph
with vertex set $G\setminus\{1\}$ whose edges are the subsets
$E\subseteq G\setminus\{1\}$ for which there exist an integer $\ell$,
with $|E|\leqslant \ell\leqslant d$, elements $h_1,\ldots,h_\ell\in E$, and signs
$a_1,\ldots,a_\ell\in\{-1,1\}$, such that every element of $E$ appears
at least once among $h_1,\ldots,h_\ell$, and
\[ h_1^{a_1}\cdots h_\ell^{a_\ell}=x.\]
If $|E|=k$, we call $E$ a $k$-edge of $\Gamma_x$. We denote by $E_k(x)$ the set of $k$-edges of $\Gamma_x$, and by $e_k(x)$ its cardinality. We also call the choices of the ordered sets $(h_1,\ldots,h_{\ell})$ and $(a_1,\ldots,a_{\ell})$ the signed ordered representation of $E$. 

Note that the definition allows walks of length at most $d$, rather than
only paths: the elements $h_1,\ldots,h_\ell$ need not be distinct. This
does not change the event that $x$ is reachable from $1$ in at most $d$
steps. Any such walk contains a path from $1$ to $x$, which is enough for our current purposes. 

We will call an expression $h_1^{a_1}\cdots h_\ell^{a_\ell}$ an
$\ell$-product. If this product witnesses that $E$ is an edge of
$\Gamma_x$, then every element of $E$ appears among
$h_1,\ldots,h_\ell$, though repetitions are allowed. Thus $|E|\leqslant \ell$.

We note  the following observations about the diameter-$d$ graphs:
\begin{itemize}
\item[(O1)] For every $k\leqslant d$ and every $g_1,\ldots,g_k \in G \setminus \{1\}$ there are at most $(2k)^{d+1}$ choices of $x$ for which the set $\{g_1,\ldots,g_k\}$ is an edge of $\Gamma_x$. 

Indeed, for every walk of length $r$ (where $k \leqslant r \leqslant d$), we have at most $k^r$ ways to choose the order of the elements and at most $2^r$ ways to choose the exponents. So we have at most
\[ (2k)^k + (2k)^{k+1} + \cdots + (2k)^d \leqslant (2k)^{d+1}\]
such choices.

\item[(O2)] Every $\Gamma_x$ has at most $2^d N^{d-1}$ edges of size $d$. 
First choose the signs $a_1,\ldots,a_d$ and the elements $g_1,\ldots,g_{d-1}$. The equation
\[g_1^{a_1}\cdots g_d^{a_d}=x\]
determines $g_d$ uniquely. Thus there are at most $2^dN^{d-1}$ signed ordered representations, and hence at most that many $d$-edges.
\begin{itemize}
\item[(O2a)] If $G$ is abelian, then we should use only one of the $d!$ ways of ordering the elements. So in this case, every $\Gamma_x$ has at most $2^d N^{d-1}/d!$ edges of size $d$.
\item[(O2b)] If $G = \mathbb{Z}_2^n$, then we should additionally only use one of the $2^d$ ways of choosing the exponents. So in this case, every $\Gamma_x$ has at most $N^{d-1}/d!$ edges of size $d$.
\end{itemize}
\item[(O3)] Every $\Gamma_x$ has 
\[ e_d(x) \geqslant (1+o(1))\frac{N^{d-1}}{d!}.\] 
 In order to give a lower bound, we restrict ourselves to products of the form $g_1 \cdots g_{d-2}g_{d-1}^{-1} g_d = x$ (i.e., with all exponents except that on $g_{d-1}$ positive). 
We first choose $g_1,\ldots,g_{d-1}$ from $G$. The choices for which
$g_i=1$ for some $i$, or $g_i=g_j$ for some $i<j<d$ contribute only $O_d(N^{d-2})$ tuples, so they may be discarded. For the remaining choices, $g_d$ is determined by
$g_d=g_{d-1}g_{d-2}^{-1}\cdots g_1^{-1}x$.
If $g_d=1$ or $g_d=g_i$ for some $i<d-1$, then one degree of freedom is lost.
If $g_d=g_{d-1}$, then $g_{d-2}^{-1}\cdots g_1^{-1}x=1$,
equivalently $g_1\cdots g_{d-2}=x$, again giving only $O_d(N^{d-2})$ tuples
after varying $g_{d-1}$.  So, we have at least $N^{d-1}-O_d(N^{d-2})$ valid ordered tuples. Dividing by $d!$ to ignore the order in which the elements appear (as the same set of elements but with the product taken in a different order could in some cases lead to the same value of product), we get the claim.

\begin{itemize}
    \item[(O3a)] If $G$ is cyclic and $x\neq x^{-1}$, then every $\Gamma_x$ has at least
    \[
        \frac{    2^d\left((N-1)^{d-1}-(1+2\binom d2)N^{d-2}\right)
        -2^d(2^d-1)N^{d-2}}{d!}
    \]
    edges of size $d$. In particular,
    \[
        e_d(x)\geqslant (1+o(1))\frac{2^dN^{d-1}}{d!}.
    \]
    If $x=x^{-1}$, then the same argument gives
    \[
        e_d(x)\geqslant (1+o(1))\frac{2^{d-1}N^{d-1}}{d!}.
    \]

Indeed, fix a choice of signs $a_1,\ldots,a_d\in\{-1,1\}$. The
argument in (O3) gives at least
\[
    (N-1)^{d-1}-\left[1+2\binom{d}{2}\right]N^{d-2}
\]
ordered $d$-tuples $(g_1,\ldots,g_d)$ with distinct entries such that
$g_1^{a_1}\cdots g_d^{a_d}=x.$
Thus, before accounting for tuples counted by more than one sign pattern,
we get $2^d$ times this many ordered tuples.

It remains to subtract sign-collision tuples. Fix two distinct sign
vectors $\mathbf{a}=(a_1,\ldots,a_d)$ and $\mathbf{b}=(b_1,\ldots,b_d)$ in $\{-1,1\}^d$. Writing $G$ additively, the two corresponding equations are
\[a_1g_1+\cdots+a_dg_d=x \quad \text{and} \quad b_1g_1+\cdots+b_dg_d=x.\]
If $\mathbf b=-\mathbf a$, then these equations imply $x=-x$, which is
impossible when $x\neq x^{-1}$. Otherwise, since $\mathbf b\neq\mathbf a$,
there exist indices $j,k$ such that
\[a_jb_k-a_kb_j=\pm 2.\]
After fixing all variables except $g_j$ and $g_k$, the two displayed
linear equations determine $g_j,g_k$ up to at most two possibilities,
because the determinant is $\pm2$ and $G$ is cyclic. Hence each pair of
distinct sign vectors contributes at most $2N^{d-2}$ collision tuples.
Summing over the $\binom{2^d}{2}$ pairs of sign vectors gives at most
\[2\binom{2^d}{2}N^{d-2}=2^d(2^d-1)N^{d-2}\]-collision tuples. Dividing by $d!$ gives the first claim.

If $x=x^{-1}$, then opposite sign vectors give the same equation. In this
case we choose one sign vector from each pair $\{\mathbf a,-\mathbf a\}$,
giving $2^{d-1}$ sign patterns, and the same argument gives the stated
bound with $2^{d-1}$ in place of $2^d$.
\end{itemize}

\end{itemize}

When proving our results, it will be crucial when dealing with dependencies to have a count on the number of pairs of edges in $\Gamma_x$ that intersect.

\begin{lemma}\label{overlap}
    Let $x\in G\setminus\{1\}$, and let $1\leqslant i \leqslant d-1$. The number of ordered pairs $(e,f)$ of $d$-edges of $\Gamma_x$ such that  $|e\cap f|=i$ is at most
\[  2^{2d}{\binom di}^2 i! N^{2d-i-2}.\]
\end{lemma}
\begin{proof}
    To see this, choose a signed ordered representation of $e$, then choose the
    $i$ shared positions in the representation of $e$, the $i$ common positions and
    their order in the representation of $f$, the signs in the representation of $f$, and finally the remaining $d-i-1$ free entries of $f$. The last
    entry of $f$ is then determined by the equation defining membership in $\Gamma_x$.
\end{proof}

When proving our results, it will also be crucial for reducing dependencies to avoid dealing with diameter-$d$ graphs $\Gamma_x$ and $\Gamma_y$ which have `many' common $d$-edges. Not only that, but in fact for every $t \leqslant r,s \leqslant d$, we would like to avoid dealing with diameter-$d$ graphs $\Gamma_x$ and $\Gamma_y$ which have `many' pairs $(e,f)$ such that $e$ is an $r$-edge of $\Gamma_x$, $f$ is an $s$-edge of $\Gamma_y$, and $e$ and $f$ intersect in $t$ elements.

To this end, for $\delta > 0$ we define the following $\delta$-dependency graph, or just dependency graph $H$:

Its vertices are the elements of $G \setminus \{1\}$, where we join $x$ to $y$ if and only if the following occurs:
\begin{itemize}
    \item[$\bullet$] For some $d \geqslant r,s \geqslant t \geqslant 1$, there are at least $N^{r+s-t-1-\delta}$ pairs $(e,f)$ where $|e|=r,|f|=s, e$ is an edge of $\Gamma_x, f$ is an edge of $\Gamma_y$, and $|e \cap f| = t$. 
\end{itemize}

We note for later use that from now on we will be working with the $\delta$-dependency graph with $0<\delta<\min\left\{\frac{\varepsilon}{4},\frac{1}{4d}\right\}$ but $N^{\delta} \geqslant 4^{d+2} d^{d+3}$. We need though to check that such a choice of $\delta$ is indeed possible. For bounded $d$, any $\delta < \varepsilon/4$ is valid as long as $N$ is large enough. For unbounded $d$, choosing $\delta = (1-\gamma)/4d_N$ works. Indeed, $0<\delta<\min\left\{\frac{\varepsilon}{4},\frac{1}{4d}\right\}$ is clearly satisfied, while
\[ \delta \log{N} = \frac{(1-\gamma)\log{N}}{4d_N} \sim \frac{d_N\log{d_N}}{1-\gamma} \geqslant (d_N+2) \log{4} + (d_N+3)\log{d_N}\]
if $N$ is large enough, satisfying the second property.

The following result says that $H$ has a `large' independent set. Later on, we will be working only within this independent set.

\begin{lemma}\label{Independent set I}
For $\delta > 0$ small enough (depending on $d$), the dependency graph $H$ has an independent set $I$ of size at least $N^{1-3\delta}$.  Furthermore, for each $x \in I$, and each $1 \leqslant k \leqslant d$, the hypergraph $\Gamma_x$ has at most $N^{k-1+\delta}$ edges of size $k$. 
\end{lemma}

\begin{proof}
For each $1 \leqslant k \leqslant d$, there are at most $(2k)^{d+1} N^k$ $k$-edges in all of the $\Gamma_x$'s: There are at most $N^k$ ways to pick the elements of the $k$-edge, and for each choice, by (O1), there are at most $(2k)^{d+1}$ choices  of $x$ for which this is a $k$-edge of $\Gamma_x$. 

It follows that at most $(2k)^{d+1} N^{1-\delta}$ of the $\Gamma_x$'s have more than $N^{k-1+\delta}$ edges of size $k$. We remove all such vertices $x$ for every $1 \leqslant k \leqslant d$ and we are left with a subgraph $H'$ of $H$ of size at least $N/2$.

We now fix a vertex $x$ of $H'$, and $r,s,t$ satisfying $1 \leqslant t \leqslant r,s \leqslant d$. We count the pairs $(e,f)$ where $e$ is an edge of $\Gamma_x$ of size $r$, $f$ is an edge of $\Gamma_y$ of size $s$ for some $y$, and $|e \cap f| = t$:

By the definition of $H'$, there are at most $N^{r-1+\delta}$ choices for $e$. We now have $\binom{r}{t}$ ways to choose $e \cap f$ and at most $N^{s-t}$ ways to choose the remaining elements of $f$. For each such $f$, by observation (O1), there are at most $(2s)^{d+1}$ choices for $y$, such that $f$ is an edge of $\Gamma_y$. So in total there are at most $\binom{r}{t}(2s)^{d+1}N^{r+s-t-1+\delta}$ such pairs.

So there are at most $\binom{r}{t}(2s)^{d+1}N^{2\delta}$ elements $y$ for which there are at least $N^{r+s-t-1-\delta}$ pairs $(e,f)$ where $|e|=r,|f|=s, e$ is an edge of $\Gamma_x, f$ is an edge of $\Gamma_y$, and $|e \cap f| = t$. 

So by the definition of $H$, the vertex $x$ (of $H'$) has degree at most
\[ \sum_{r,s,t} \binom{r}{t}(2s)^{d+1}N^{2\delta} \leqslant \sum_{r,s} 2^r (2s)^{d+1}N^{2\delta} \leqslant d2^{d+1}(2d)^{d+1}N^{2\delta} \]
in $H$. So 
\[\Delta(H') \leqslant d2^{d+1}(2d)^{d+1}N^{2\delta} \leqslant \frac{N^{3\delta}}{4}\] 
and therefore $H'$ has an independent set $I$ of size at least
\[ \frac{|V(H')|}{2\Delta(H')} \geqslant \frac{(N/2)}{2d2^{d+1}(2d)^{d+1}N^{2\delta}} \geqslant N^{1-3\delta}\]
as required.
\end{proof}

In some special cases, we can guarantee that there are even less pairs $(e,f)$ than the ones guaranteed by the previous lemma.

\begin{lemma}\label{Independent set I - special case}
Given vertices $x,y$ of $I$ and $r,s,t$ satisfying $1 \leqslant t \leqslant r,s \leqslant d$ and $r+s-t \geqslant d+1$, then the number of pairs $(e,f)$ with $|e|=r,|f|=s, e$ is an edge of $\Gamma_x, f$ is an edge of $\Gamma_y$, and $|e \cap f| = t$ is at most $N^{r+s-t-2+2\delta}$.  
\end{lemma}

\begin{proof}
Without loss of generality we may assume that $r \geqslant s$. 

Since $I$ is a subset of vertices of $H'$, there are at most $N^{s-1+\delta}$ choices for $f$. 

It remains to choose the remaining $r-t$ elements of $e$. Consider an $\ell$-product witnessing that $e$ is an edge, where $r\leqslant \ell\leqslant d$.
Since $r+s-t\geqslant d+1\geqslant \ell+1$, at least one element of $e\setminus f$
appears exactly once in this product. Otherwise the product would have
length at least
\[
    2(r-t)+t\geqslant r+s-t\geqslant d+1,
\]
a contradiction. Fix such an element. There are $\binom{s}{t}$ choices
for $e\cap f$, at most $r-t$ choices for the element of $e\setminus f$
which appears exactly once, $\ell$ choices for its position in the
$\ell$-product, $2^\ell$ choices for the signs, and at most
$(r-1)^{\ell-1}$ choices for which of the remaining elements of $e$
appear in the other positions. Finally, the remaining $r-t-1$ elements of
$e\setminus f$ can be chosen in at most $N^{r-t-1}$ ways, and the
uniquely occurring element is then determined by the equation defining
membership in $\Gamma_x$. Summing over $r\leqslant \ell\leqslant d$, for each fixed
$f$ there are at most $C_{d,r,s,t}N^{r-t-1}$
possible choices of $e$, where
\[C_{d,r,s,t}= \binom{s}{t}(r-t)\sum_{\ell=r}^{d}\ell\,2^\ell(r-1)^{\ell-1}\leqslant 2^d d^{d+3} \leqslant N^{\delta}.
\]
Therefore the number of such pairs $(e,f)$ is at most $N^{r+s-t-2+2\delta}$ for \(N\) sufficiently large.
\end{proof}

\section{Proofs of Theorems~\ref{T:D>=3}, \ref{T:D>=3-Abelian}, and~\ref{E:D>=3}}\label{sec:lower-bound}

Let $I$ be the independent set obtained in Lemma~\ref{Independent set I},
let $x\in I$, and let $B_x$ denote the event that $x$ is at distance
greater than $d$ from the identity. This gives 
\[
    B_x=\bigcap_{e\in \Gamma_x}\overline{A_e},
\]
where $A_e$ is the event that all elements of $e$ appear in the generating
set.

Recall that, for $1 \leqslant i \leqslant d$, we write $E_i(x)$ for the edges of $\Gamma_x$ of size $i$, and $e_i(x)=|E_i(x)|$.

By the definition of $I$ we have $e_i(x) \leqslant N^{i-1+\delta}$ for each $1 \leqslant i \leqslant d$.  By (O2), (O2a), and (O2b) we also have $e_d(x) \leqslant C(d) N^{d-1}$, where $C(d) = 2^d$ for a generic group (for Theorem~\ref{T:D>=3}), $C(d) = 2^d/d!$ if $G$ is abelian (for Theorem~\ref{T:D>=3-Abelian}), and  $C(d) = 1/d!$ if $G = \mathbb{Z}_2^n$ (for Theorem~\ref{E:D>=3}).

By Kleitman's Inequality,
\[ \Pr(B_x) \geqslant \prod_{i=1}^d (1-p^i)^{e_i(x)}\]
For each $1\leqslant i \leqslant d-1$ we have $p^i \to 0$ and
\[ p^ie_i(x) \leqslant (d!\log{N})^{i/d} N^{i/d-1+\delta}\]
So for $\delta > 0$ small enough, we have $p^ie_i(x) \to 0$ for each $1 \leqslant i \leqslant d-1$. So by Lemma~\ref{exponential asymptotic} we get $(1-p^i)^{e_i(x)} \sim 1$ for each $1 \leqslant i \leqslant d-1$.

We also have 
\[ p^{2d}e_d(x) \leqslant (d!\log{N})^2 C(d)N^{1-d} \to 0\]
so again by Lemma~\ref{exponential asymptotic} we get
\[ (1-p^d)^{e_d(x)} \sim \exp\{-p^de_d(x)\} \geqslant \exp\{-(1-\varepsilon)\log{N}\} = N^{-1+\varepsilon}. \]
This is because for all three cases the value of $C(d)$ leads to $p^de_d(x) \leqslant (1-\varepsilon)\log{N}$. 

We therefore get
\[\Pr(B_x) \geqslant (1+o(1))N^{-1+\varepsilon}. \]
Let $X$ be the number of vertices of $I$ which are at distance greater than $d$ from the identity. Then
\[ \mathbb{E}X = \sum_{x \in I} \Pr(B_x) \geqslant (1+o(1)) N^{1-3\delta} \cdot N^{-1+\varepsilon}\to +\infty\]
since $3\delta <\varepsilon$. 

Given distinct $x,y\in I$, for each $1 \leqslant i \leqslant d$, we write $e_i(x\vee y)$ for the number of edges of $\Gamma_x \cup \Gamma_y$ of size $i$ and $e_i(x\wedge y)$ for the number of common edges of $\Gamma_x$ and $\Gamma_y$ of size $i$. So
\[ e_i(x) + e_i(y) = e_i(x \vee y) + e_i(x \wedge y)\] 
for each $i$. 

For $x\in I$, let $\mathcal D_x=E_d(x)$ be the set of $d$-edges of
$\Gamma_x$, and define
\[
    D_x=\bigcap_{e\in\mathcal D_x}\overline{A_e}.
\]
Since $B_x$ requires avoiding all edges of $\Gamma_x$, while $D_x$
requires avoiding only the $d$-edges, we have $B_x\subseteq D_x$. Thus, in
order to upper-bound $\Pr(B_x\cap B_y)$, it is enough to consider only
$d$-edges.
Thus, for distinct $x,y\in I$,
\[
    \Pr(B_x\cap B_y)\leqslant \Pr(D_x\cap D_y).
\]

For distinct $x,y\in I$, we apply Janson's inequality to the family
$\mathcal D_x\cup\mathcal D_y$. We have
\[
    |\mathcal D_x\cup\mathcal D_y|
    =
    e_d(x)+e_d(y)-e_d(x\wedge y).
\]
Hence
\[ \Pr(D_x\cap D_y) \leqslant \exp\left\{-p^d(e_d(x)+e_d(y)-e_d(x\wedge y)) +\frac{1}{2}\Delta_{x,y}\right\},\]
where
\[\Delta_{x,y}  =\sum_{\substack{e,f\in\mathcal D_x\cup\mathcal D_y\\e\neq f,\ e\cap f\neq\emptyset}}\Pr(A_e\cap A_f).\]

We now split $\Delta_{x,y}$ into internal and cross contributions:
\[ \Delta_{x,y} = \Delta_x^{\mathrm{int}} +  \Delta_y^{\mathrm{int}} + \Delta_{x,y}^{\mathrm{cross}}.\]

By Lemma~\ref{overlap}, the internal contributions are at most
\begin{align*}
\Delta_x^{\mathrm{int}}+\Delta_y^{\mathrm{int}} &\leqslant 2\sum_{i=1}^{d-1} 2^{2d}\binom di^2 i!N^{2d-i-2}p^{2d-i}  \\    
&\leqslant 2 \cdot 4^d \cdot d!^2 \cdot dN^{2d-3}p^{2d-1}.
\end{align*}
We claim that this is $o(1)$. To check this, it is enough, after taking logarithms, to check that
\[(2d-1)\log{(pN)} + 2\log{d!} - 2\log{N} \to -\infty. \]
Since $(pN)^d \leqslant d!N\log{N}$ and $d! \leqslant d^d$, it is enough to check that
\[ 4d_N\log{d_N} + \left(2 - \frac{1}{d_N}\right)(\log{N}+\log{\log{N}})-2\log{N} \to -\infty.\]
This is true since $4d_N^2\log{d_N} \sim (1-\gamma)^2\log{N}$ and $d_N \log{\log{N}} = o(\log{N})$.

For the cross contribution, suppose $e\in\mathcal D_x$, $f\in\mathcal D_y$, and $|e\cap f|=i$. Since $e\neq f$, then $1\leqslant i\leqslant d-1$, so $|e\cup f|=2d-i\geqslant d+1$. Since $x,y\in I$ by Lemma~\ref{Independent set I - special case} we have at most $N^{2d-i-2+2\delta}$ such cross-pairs. Therefore
\[\Delta_{x,y}^{\mathrm{cross}} \leqslant \sum_{i=1}^{d-1}  N^{2d-i-2+2\delta}p^{2d-i} = N^{2d-2+2\delta}p^{2d}\sum_{i=1}^{d-1} (pN)^{-i} \leqslant N^{2d-2+2\delta}p^{2d} \cdot \frac{2}{pN}
\]
where we used the fact that $pN \to \infty$. We claim that this is $o(1)$. To check this, it is enough, after taking logarithms, to check that 
\[(2d-1)\log{(pN)} + (2\delta-2)\log{N} \to -\infty. \]
Since $(pN)^d \leqslant d!N\log{N}$, then
\begin{align*}
(2d-1)\log{(pN)} + (2\delta-2)\log{N} &\leqslant \left(2 - \frac{1}{d_N} \right)\left(\log{d_N!} + \log{\log{N}}  \right) +  \left(2\delta-\frac{1}{d_N}\right)\log{N}\\
&\leqslant (2d_N-1)\log{d_N} +2\log{\log{N}} - \frac{\log{N}}{2d_N} \to -\infty
\end{align*}
since $4d_N^2\log{d_N} \sim (1-\gamma)^2\log{N}$ and $d_N \log{\log{N}} = o(\log{N})$.

Thus $\Delta_{x,y}=o(1)$, and so
\[\Pr(D_x\cap D_y)\leqslant    (1+o(1))  \exp\left\{ -p^d(e_d(x)+e_d(y)-e_d(x\wedge y))\right\}.\]
Since $x$ and $y$ are non-adjacent in $H$, applying the definition of $H$ with $r=s=t=d$ gives $e_d(x\wedge y)\leqslant N^{d-1-\delta}$. Therefore
\[
    p^d e_d(x\wedge y)
    \leqslant  (d!\log N)N^{-\delta}   =    o(1).
\]

From the lower bound on $\Pr(B_x)$ obtained above, we have
\[
    \Pr(B_x)  \geqslant    (1+o(1))\exp\{-p^de_d(x)\},
\]
and similarly
\[
    \Pr(B_y)     \geqslant     (1+o(1))\exp\{-p^de_d(y)\}.
\]
Hence
\[
    \Pr(D_x\cap D_y)     \leqslant     (1+o(1))\Pr(B_x)\Pr(B_y).
\]
Since $B_x\cap B_y\subseteq D_x\cap D_y$, we conclude that
\[
    \Pr(B_x\cap B_y)   \leqslant   (1+o(1))\Pr(B_x)\Pr(B_y).
\]

The reverse inequality
\[ \Pr(B_x \cap B_y) \geqslant \Pr(B_x)\Pr(B_y)\]
is immediate by Kleitman.
Thus
\begin{align*}
\mathbb{E}X(X-1) &= \sum_{x \in I}\sum_{y \in I \setminus\{x\}} \Pr(B_x \cap B_y) \\
&\leqslant   \sum_{x \in I}\sum_{y \in I \setminus\{x\}} (1+o(1))\Pr(B_x)\Pr(B_y) \\
&= (1+o(1))(\mathbb{E}X)^2.
\end{align*} 
So by Chebyshev's Inequality,
\[ \Pr(X=0) \leqslant \frac{\Var(X)}{(\mathbb{E}X)^2} = \frac{\mathbb{E}X(X-1) + \mathbb{E}X - (\mathbb{E}X)^2}{(\mathbb{E}X)^2} = o(1) + \frac{1}{\mathbb{E}X} = o(1). \]
So with high probability $X > 0$ and so the diameter of $\Gamma$ is greater than $d$.

\section{Proofs of Theorems~\ref{T:<=D} and \ref{E:D>=3-Abelian}}

Given an edge $e$ of $\Gamma_x$, write $A_e$ for the event that all vertices of the edge $e$ appear in the generating set. Then 
\[ \Pr(d(1,x) > d) = \Pr\left(\bigcap_{e \in E(\Gamma_x)} \overline{A_e} \right) \leqslant  \Pr\left(\bigcap_{e \in E_d(x)} \overline{A_e} \right) \]

By (O3) we have $e_d(x) \geqslant (1+o(1))N^{d-1}/d!$. For the more specific case of Theorem~\ref{E:D>=3-Abelian}, we have $e_d(x) \geqslant (1+o(1))2^d N^{d-1}/d!$, if $x\ne x^{-1}$, and half that for the unique non-identity involution in a cyclic group of even order. So, except for the possible unique cyclic involution, we have 
\[\sum_{e \in E_d(x)} \Pr(A_e)=p^de_d(x)
\geqslant (1+o(1))(1+\varepsilon)\log N,\]
while if $x$ is the exceptional involution we have
\[ \sum_{e \in E_d(x)} \Pr(A_e)=p^de_d(x)
\geqslant (1+o(1))\frac{1+\varepsilon}{2}\log N.\]

Suppose now that the events $A_e$ and $A_f$ (with $e\neq f$) are not independent. This happens when $e \cap f \neq \emptyset$. Using Lemma~\ref{overlap}, we have
\[\sum_{e\cap f\ne\emptyset}\Pr(A_e\cap A_f)\leqslant \sum_{i=1}^{d-1} 2^{2d}\binom di^2 i! N^{2d-i-2}p^{2d-i} = o(1).
\]
This follows in exactly the same way as in the bounding of the internal contributions.


So by Janson, for any $x$ which is not the exceptional involution,
\[ \Pr(d(1,x) > d) \leqslant \exp\left\{-\sum_{e \in E_d(x)} \Pr(A_e) + \sum_{e \cap f \neq \emptyset} \Pr(A_e \cap A_f)\right\} \leqslant N^{-(1+\varepsilon/2)}\]
where the second sum is over all sets $\{e,f\}$ of distinct edges $e,f$  of $\Gamma_x$.

Letting $X$ be the number of vertices of $G \setminus \{1\}$ which are at distance greater than $d$ from the identity, we get 
\[\mathbb{E}X \leqslant N \cdot N^{-(1+\varepsilon/2)} +  N^{-(1+\varepsilon)/2} = o(1)\] 
and therefore by Markov's Inequality $X=0$ with high probability. So the diameter of $\Gamma$ is with high probability at most $d$.

\section{Proof of Theorem~\ref{E:D=d}}

Just like in the proof of the previous section, letting $X$ be the number of vertices of $G \setminus \{1\}$ which are at distance greater than $d$ from the identity, it would be enough to show that $\mathbb{E}X \to 0$. For this, it would be enough to prove that
\[ \sum_{e \in E_d(x)} \Pr(A_e) \geqslant (1+o(1))(1+\varepsilon)\log{N}, \]
for every $x \neq 1$. However, this is actually not true for vertices $x$ which are involutions, or which have `small' conjugacy class sizes. We will therefore find lower sums for $\sum_{e \in E_d(x)} \Pr(A_e)$ depending on whether $x$ is an involution and whether it has a `large' conjugacy class size or not. The following lemma will enable us to compute lower bounds of such sums.

\begin{lemma} Let $G$ be a group satisfying the conditions of Theorem~\ref{E:D=d}. Then
\begin{itemize}
    \item[(a)] If $x^2 \neq 1$ and $\cl(x) > M$, then there are at least $(1-\varepsilon/4)N^{d-1}$ tuples $(g_1,\ldots,g_{d-1})$ of distinct elements of $G \setminus \{1,x\}$ such that all $d!2^d$ products of the form $h_1^{a_1} \cdots h_d^{a_d}$, where $h_1,\ldots,h_d$ is a permutation of $x,g_1,\ldots,g_{d-1}$ and $a_i \in \{-1,1\}$ for each $1 \leqslant i \leqslant d$, are distinct.  
    \item[(b)] If $x^2 = 1$ and $\cl(x) > M$, then there are at least $(1-\varepsilon/4)N^{d-1}$ tuples $(g_1,\ldots,g_{d-1})$ of distinct elements of $G \setminus \{1,x\}$ such that all $d!2^{d-1}$ products of the form $h_1^{a_1} \cdots h_d^{a_d}$, where $h_1,\ldots,h_d$ is a permutation of $x,g_1,\ldots,g_{d-1}$, and $ a_i \in \{-1,1\}$ for each $1 \leqslant i \leqslant d$, and furthermore, if $x=h_i$, then $a_i=1$, are distinct. 
    \item[(c)] If $x^2 \neq 1$ and $\cl(x) \leqslant M$, then there are at least $(1-\varepsilon/4)N^{d-1}$ tuples $(g_1,\ldots,g_{d-1})$ of distinct elements of $G \setminus \{1,x\}$ such that all $(d-1)!2^d$ products of the form $h_1^{a_1} \cdots h_d^{a_d}$, where $h_1,\ldots,h_{d-1}$ is a permutation of $g_1,\ldots,g_{d-1}$, $x=h_d$, and $a_i \in \{-1,1\}$ for each $1 \leqslant i \leqslant d$, are distinct.  
    \item[(d)] If $x^2 = 1$ and $\cl(x) \leqslant M$, then there are at least $(1-\varepsilon/4)N^{d-1}$ tuples $(g_1,\ldots,g_{d-1})$ of distinct elements of $G \setminus \{1,x\}$ such that all $(d-1)!2^{d-1}$ products of the form $h_1^{a_1} \cdots h_d^{a_d}$, where $h_1,\ldots,h_{d-1}$ is a permutation of $g_1,\ldots,g_{d-1}$, $x=h_d$,  $a_i \in \{-1,1\}$ for each $1 \leqslant i \leqslant d-1$, and $a_d=1$, are distinct.  
\end{itemize}    
\end{lemma}

\begin{proof} \mbox{}
 \begin{itemize}
 \item[(a)] Suppose that two such different choices of permutations and exponents give the same product, say
\[ h_1^{a_1} \cdots h_d^{a_d} = f_1^{b_1} \cdots f_d^{b_d}\]
Suppose first that $f_1 \neq x,h_1$ and pick $i$ such that $f_1 = h_i$. If $b_1 = -a_i$, then we can rewrite this equality as $(h_i^{a_i}y)^2 = z$ where
\[ y = h_1^{a_1} \cdots h_{i-1}^{a_{i-1}} \quad \text{and} \quad z = f_2^{b_2} \cdots f_d^{b_d}h_d^{-a_d} \cdots h_{i+1}^{-a_{i+1}}y.\]
We now count how many times this situation can occur. Picking the values of all $g$'s apart from $f_1 = h_i$, then $y$ and $z$ are completely determined. By Theorem~\ref{main-rep} the equation has at most $\sqrt{N\cl(G)}$ solutions for $h_i^{a_i}y$ and therefore at most that many for $h_i$. Since $M \leqslant N^{1/2} \leqslant N^{1-\frac{1}{d}}$, then by
Lemma~\ref{cl(G)-bound} there are at most 
 \[\sqrt{N\left(N^{1/d}+\frac{N}{M}\right)} \leqslant 
 \sqrt{\frac{2N^2}{M}} < \frac{2N}{\sqrt{M}}\]
such solutions. So there are at most 
\[\frac{2(d!2^d)^2 N^{d-1}}{\sqrt{M}}\]
situations for which this occurs where the extra factor is to account for all possibilities of forming the products.

Suppose now that $f_1 \neq x,h_1$ but $b_1 = a_i$. Then we can rewrite this equality as $h_i^{-a_i}yh_i^{a_i} = z$ where
\[ y = h_1^{a_1} \cdots h_{i-1}^{a_{i-1}} \quad \text{and} \quad z =f_2^{b_2} \cdots f_d^{b_d}h_d^{-a_d} \cdots h_{i+1}^{-a_{i+1}}. \]
We now count how many times this situation can occur. Picking the values of all $g$'s apart from $f_1 = h_i$, then $y$ and $z$ are completely determined. If $i=2$ and $h_1 = x$ then $y = x$ or $y=x^{-1}$ and since $\cl(x) \geqslant M$, by Theorem~\ref{commuting_pairs} at most $N/M$ choices of the value of $h_i$ lead to equality. Otherwise, at least one of the $h$'s in the product defining $y$ is one of the $g$'s and since the $h_1,\ldots,h_{i-1}$ are all distinct, then, while varying the choices of the $g$'s (apart from $f_1=h_i$), even allowing some of them to be equal, all elements of $G$ appear exactly the same number of times as $y$, namely $N^{d-3}$ times each. So at most $N^{d-2}/M$ times do we get a $y$ with $\cl(y) \leqslant M$ and therefore, accounting for the choice of the value of $f_1=h_i$, we have at most $N^{d-1}/M$ situations where this occurs. In the at most $N^{d-2}$ times where $\cl(y) > M$, just like in the case where $y = x,x^{-1}$, at most $N/M$ choices of the value of $h_i$ lead to equality. In total, in this case, the equality situation occurs at most $2N^{d-1}/M$ times.

Suppose now that $f_1=x$ or $f_1 = h_1$. If $f_1 = h_1$ and $b_1 = -a_1$ then we get an equation of the form $(h_1^{a_1})^2 = y$ which can be treated as in the previous paragraphs. If $f_1 = h_1$ and $b_1=a_1$ we `cancel' them, while still allowing them to take all possible $N$ values. We repeat the argument with $f_2$ and $h_2$ and so on. Each time, as long as we do not first reach an $f_i$ equal to $x$, this situation can occur the same number of times as the count in the previous paragraphs.

Suppose now that $f_i = x$, after we have cancelled $f_1,\ldots,f_{i-1}$ (including the possibility that $i=1$). In this case we start working backwards from $f_d$ in the same manner. The only possible problem is if we have to cancel $f_{d},f_{d-1},\ldots,f_{i+1}$ but this cannot occur as it would mean that the products were identical from the start without the need to choose the values of the $g$'s. 

So overall, taking into account that there are at most $d^2$ choices on how many $f$'s we cancelled from the left and from the right, we have at most $2d^2 (d!2^d)^2N^{d-1}/\sqrt{M}$ situations where equality occurs. So we have at least 
\[ (N-1)(N-2) \cdots (N-(d-1)) - \frac{2d^2 (d!2^d)^2 N^{d-1}}{\sqrt{M}} \]
tuples $(g_1,\ldots,g_{d-1})$ giving distinct products. Note that for $x\in [0,1/2]$ we have
$(1-x)(1+2x) = 1+x-2x^2 \geqslant 1$ and therefore
\[ 1-x \geqslant \frac{1}{1+2x} \geqslant \frac{1}{e^{2x}}.\]
So since $d^2 \leqslant d_N^2 = o(N)$ we get
\begin{align*} 
(N-1)(N-2) \cdots (N-(d-1)) &= \left(1 - \frac{1}{N}\right) \cdots \left(1 - \frac{d-1}{N}\right)N^{d-1} \\
&\geqslant \exp\left\{-\frac{2(1 + 2 + \cdots + (d-1))}{N} \right\}N^{d-1} \\
&= \exp\left\{-\frac{d(d-1)}{N} \right\}N^{d-1} = (1+o(1))N^{d-1}.
\end{align*}
For fixed $\varepsilon>0$ and $N$ large enough, we also have
\[ \frac{16d^2(d!2^d)^2}{\varepsilon} < (d_N!)^3 < d_N^{3d_N}\]
therefore
\[ \log{\left(\frac{16d^2(d!2^d)^2}{\varepsilon}\right)} < \frac{5d_N \log{d_N}}{2} < \frac{5}{2}\sqrt{\frac{\log{N}}{2\log{\log{N}}}} \cdot \frac{\log{\log{N}}}{2} < \frac{\log{M}}{2}.\]
It follows that 
\[\frac{2d^2 (d!2^d)^2 N^{d-1}}{\sqrt{M}} \leqslant \frac{\varepsilon N^{d-1}}{8} \]

\item[(b)] This is similar to (a) since restricting the exponent of $x$ to always be equal to $1$ guarantees that the situation $x^2=z$ never occurs. (The problem here would only occur if cancellations would force $z$ to be equal to $1$.)
\item[(c)] Similar to (a). Note that the situation $h_i^{-a_i}yh_i^{a_i} = z$ with $y=x$ never arises.
\item[(d)] Similar to (a). Note that the situation $h_i^{-a_i}yh_i^{a_i} = z$ with $y=x$ never arises. Nor does the situation $x^2=z$. \qedhere  
\end{itemize}   
\end{proof}

The above Lemma guarantees that if $x^2 \neq 1$ and $\cl(x) > M$, then $e_d(x) \geqslant 2^d(1-\varepsilon/4)N^{d-1}$. Indeed, for every tuple $(g_1,\ldots,g_{d-1})$ given in part (a) of the Lemma, and every one of the $d!2^d$ products $h_1^{a_1} \cdots h_d^{a_d}$, we define $g_d$ to be the value of this product. Solving for $x$, we can write $x$ as a $d$-product of the set $\{g_1,\ldots,g_d\}$. The fact that all products are distinct guarantees that we obtain at least $2^d(1-\varepsilon/4)N^{d-1}$ $d$-edges $\{g_1,\ldots,g_d\}$ of $\Gamma_x$. Here, we divided by $d!$ to move from tuples to sets. We note that when an element such as $g_d$ is uniquely determined by the remaining variables, it might happen that $g_d=1$ or $g_d = g_i$ for some $i<d$. However, each such condition imposes an additional relation on the remaining $d-1$ free variables, reducing the number of such configurations to at most $O(N^{d-2})$. These invalid configurations are therefore absorbed in the error term and do not affect the main estimate.
 
Similarly, the lemma guarantees that $e_d(x) \geqslant 2^{d-1}(1-\varepsilon/4)N^{d-1}$ if $x^2 = 1$ and $\cl(x) > M$, $e_d(x) \geqslant \frac{2^d}{d}(1-\varepsilon/4)N^{d-1}$ if $x^2 \neq 1$ and $\cl(x) \leqslant M$, and $e_d(x) \geqslant \frac{2^{d-1}}{d}(1-\varepsilon/4)N^{d-1}$ if $x^2 = 1$ and $\cl(x) \leqslant M$.

Thus
\[ \sum_{e \in E_d(x)} \Pr(A_e) = p^de_d(x) \geqslant \begin{cases}
    (1-\varepsilon/4)(1+\varepsilon)\log{N} & \text{if $x^2 \neq 1$ and $\cl(x) > M$} \\
   \frac{1}{2}(1-\varepsilon/4)(1+\varepsilon)\log{N} & \text{if $x^2 = 1$ and $\cl(x) > M$} \\
    \frac{1}{d}(1-\varepsilon/4)(1+\varepsilon)\log{N} & \text{if $x^2 \neq 1$ and $\cl(x) \leqslant M$}\\
    \frac{1}{2d}(1-\varepsilon/4)(1+\varepsilon)\log{N} & \text{if $x^2 = 1$ and $\cl(x) \leqslant M$}
\end{cases} \]

The number of elements $x$ of each type is enough to guarantee that $\mathbb{E}X = o(1)$ and therefore we can conclude as in the previous section.

\end{document}